\documentclass{amsart}
\usepackage{amssymb,epsfig}
\usepackage{ifpdf}
\usepackage{euscript}

\usepackage{color}
\usepackage{bm}
\usepackage{bbold}
\usepackage{mathrsfs}

\newtheorem{theorem}{Theorem}[section]
\newtheorem{lemma}[theorem]{Lemma}
\newtheorem{proposition}[theorem]{Proposition}
\newtheorem{corollary}[theorem]{Corollary}

\theoremstyle{definition}

\theoremstyle{remark}
\newtheorem{remark}[theorem]{Remark}

\numberwithin{equation}{section}

\newcommand{\eps}{\varepsilon}

\newcommand{\R}{\mathbb R}
\newcommand{\N}{\mathbb N}

\newcommand{\VV}{\mathscr{V}}

\newcommand{\ZZ}{\mathscr{Z}}

\newcommand{\HH}{\mathscr{H}}

\newcommand{\cL}{\mathcal L}

\newcommand{\identity}{\mathrm{Id}}

\newcommand{\be}{\begin{equation}}
\newcommand{\ee}{\end{equation}}

\newcommand{\tria}{\mathcal T}

\DeclareMathOperator{\vol}{vol}
\DeclareMathOperator{\ran}{ran}

\DeclareMathOperator{\clos}{clos}
\DeclareMathOperator{\supp}{supp}

\DeclareMathOperator{\Span}{span}

\usepackage{scalerel}

\title{A quadratic finite element wavelet Riesz basis}
\thanks{The first author has been supported by the Netherlands Organization for Scientific Research
(NWO) under contract. no. 613.001.216}

\date{\today}

\author{Nikolaos Rekatsinas and Rob Stevenson}
\address{
Korteweg-de Vries Institute for Mathematics,
University of Amsterdam,
P.O. Box 94248,
1090 GE Amsterdam, The Netherlands}
\email{n.rekatsinas@uva.nl, r.p.stevenson@uva.nl}



\begin{document}

\begin{abstract} 
In this paper, continuous piecewise quadratic finite element wavelets are constructed on general polygons in $\R^2$. The wavelets are stable in $H^s$ for $|s|<\frac{3}{2}$ and have two vanishing moments. Each wavelet is a linear combination of 11 or 13 nodal basis functions.
Numerically computed condition numbers for $s \in \{-1,0,1\}$ are provided for the unit square.
\end{abstract}

\maketitle

\section{Introduction}
It is well-known that, properly scaled, an infinite countable collection of wavelets can generate Riesz bases for a scale of Sobolev spaces.
Such wavelets can be applied for the numerical solution of PDEs and singular integral equations.
For suitable wavelets, the (infinite) stiffness matrix of such an operator equation is boundedly invertible, so that the residual of an approximation is equivalent to its error, meaning that this residual can be used as an a posteriori error estimator for driving an adaptive algorithm.
Using that wavelets have vanishing moments, even for singular integral equations the matrix is close to being sparse, and therefore its application can be efficiently approximated.
Furthermore, in any case for elliptic equations, any principal submatrix is uniformly well-conditioned allowing for an efficient iterative solution of the arising Galerkin systems.

As with most wavelet applications, it is important that the wavelets have local supports.  Other than with classical wavelet applications,  as data compression and signal analysis,
for solving operator equations the corresponding dual wavelets do not enter the computations, and their support sizes are irrelevant. This induces a lot of freedom in the construction of suitable wavelet bases.
For more information on the application of wavelets for the (adaptive) solution of operator equations, we refer to \cite{55,249.92,298} and the references cited there.

Traditionally, wavelets are constructed on the line or on the interval $[0,1]$.
Then the application of a tensor product construction yields wavelets on $\R^n$ or $[0,1]^n$.
For equipping more general domains in $\R^n$, or their boundaries,  with wavelet bases domain decomposition techniques have been developed (e.g. \cite{35.95,54.5,54.6}).

Another approach to treat non-product domains is to construct wavelets in finite element spaces w.r.t. a nested sequence of meshes.
This approach, to which also this paper is devoted, inherits the full flexibility from the finite element method concerning the shape of the domain.

To realize the Riesz basis property, one can rely on the theory of biorthogonal space decompositions (\cite{53}).
It starts with two multiresolution analyses $(V_j)_{j \geq 0}$, $(\tilde{V}_j)_{j \geq 0}$ on the given domain that both satisfy Jackson and Bernstein estimates, and for which $V_j$ and $\tilde{V}_j$ are relatively close in the sense that they satisfy inf-sup conditions, uniformly in $j$. Then for each $j$, one constructs the wavelets on `level' $j$ as a basis for the biorthogonal complement $V_j \cap \tilde{V}_{j-1}^{\perp_{L_2}}$ of $V_{j-1}$ in $V_j$.
The union over $j$ of such wavelets form, properly scaled, a Riesz basis for the Sobolev space with smoothness index $s$ for $s$ in an interval $(s_{\min},s_{\max}) \ni 0$ determined by the aforementioned Jackson and Bernstein estimates. With finite element wavelets, both multiresolution analyses are sequences of finite element spaces w.r.t. a common sequence of meshes.
Linear finite elements of this type, with $\tilde{V}_j=V_j$, were constructed in \cite{180,249.71,75.2,138.5}, and higher order ones, also with $\tilde{V}_j\neq V_j$, can be found in \cite{56,239.17}. With the exception of \cite{249.71}, the constructions in these references were restricted to two space dimensions.

An alternative possibility is to relax `true' biorthogonality w.r.t. $L_2$ to an approximate biorthogonality, for instance biorthogonality w.r.t. to a level-dependent, approximate $L_2$-scalar product.
Usually the resulting wavelet bases have smaller supports than those than span truly $L_2$-biorthogonal complements, but on the other hand typically their $s_{\min}$ is larger and sometimes positive. 
Linear finite element wavelets of this type were can be found in  \cite{306,249.70,191}, and quadratics ones in \cite{187}.

In \cite{243.86}, we applied an adaptive wavelet method for solving time-dependent parabolic PDEs in a simultaneous space-time variational formulation.
One of the arising spaces that had to be equipped with a wavelet Riesz basis is the intersection of the Bochner spaces $L_2((0,T);H_0^1(\Omega)) \cap H^1((0,T);H^{-1}(\Omega))$, where $\Omega \subset \R^2$ denotes the spatial domain. We equipped this space with a basis of tensor products of temporal wavelets and spatial wavelets. In order to obtain a Riesz basis for the aforementioned intersection space, the collection of temporal wavelets has to be, properly scaled, a Riesz basis for $L_2(0,T)$ and for $H^1(0,T)$, which wavelets are amply available, whereas the 
 collection $\Psi$ of spatial wavelets has to be, properly scaled, a Riesz basis for $H^1_0(\Omega)$ and for its dual $H^{-1}(\Omega)$.
 Moreover, in order to obtain a sufficient near-sparsity of the resulting stiffness matrix, we need these wavelets from $\Psi$ to have 2 vanishing moments, or more precisely, cancellation properties of order 2.
 Finally, since we wrote the PDE as a first order system least squares system, and a suitable wavelet basis for the flux variable is at least of order 2,  we need the wavelets from $\Psi$ to be of order 3, i.e. piecewise quadratics.

Piecewise quadratics wavelets in finite element spaces were constructed in \cite{56,187,239.17}. Those in \cite{187} have small supports, but are not stable in $H^s$ for $s \leq 0$ and do not have vanishing moments.
The quadratic wavelets in the other two references satisfy our needs, and have even 3 vanishing moments ($\tilde{V}_j=V_j$). Unfortunately the condition numbers of those from \cite{56} turned out to be very large. For that reason, in \cite{239.17} we applied the available freedom in the general construction from \cite{56} to arrive at wavelets that are much better conditioned. 
The price to be paid was an increased support size. The continuous piecewise quadratic finite element wavelets from \cite{239.17} are linear combinations of 87 nodal basis functions.
In the current work, we change $\tilde{V}_j$ into the space of continuous piecewise linears w.r.t. a dyadically refined mesh, and use the available freedom to construct piecewise quadratics wavelets
that satisfy all our needs, have condition numbers similar to those from \cite{239.17}, and are given as linear combinations of 11 or 13 nodal basis functions.
\medskip

This paper is organized as follows: In Sect.~\ref{Stheory} we review the general theory on biorthogonal space decompositions, and in Sect.~\ref{Sconstruction} we apply the general principles of the element-by-element construction of finite element wavelets to construct continuous piecewise quadratic wavelets in two space dimensions, with small supports and $2$ vanishing moments. We provide numerically computed condition numbers in $H^1$, $L_2$ and $H^{-1}$-norms.
\medskip

We will use the following notations.
By $C \lesssim D$ we will mean that $C$ can be bounded by a multiple of $D$, independently of parameters which C and D may depend on. Obviously, $C \gtrsim D$ is defined as $D \lesssim C$, and $C\eqsim D$ as $C\lesssim D$ and $C \gtrsim D$.

For normed linear spaces $\mathscr{A}$ and $\mathscr{B}$, for convenience in this paper always over $\R$,  $\cL(\mathscr{A},\mathscr{B})$ will denote the space of bounded linear mappings $\mathscr{A} \rightarrow \mathscr{B}$ endowed with the operator norm $\|\cdot\|_{\cL(\mathscr{A},\mathscr{B})}$.


For a countable set $\vee$, the norm and scalar product on $\ell_2(\vee)$ will be denoted as $\|\cdot\|$ and $\langle \,,\,\rangle$, respectively.
For real square matrices $A$ and $B$ of the same size, we write $A \geq B$ when for all real vectors $x$, $\langle A x,x \rangle \geq \langle B x,x \rangle$.

A countable collection of functions $\Sigma$ will be formally viewed as a column vector. Then for a sequence of scalars ${\bf c}=(c_\sigma)_{\sigma \in \Sigma}$, we set ${\bf c}^\top \Sigma:=\sum_{\sigma \in \Sigma} c_\sigma \sigma$.
For countable collections of functions $\Sigma$ and  $\Phi$ in a Hilbert space $\HH$, we define the (formal) matrix $\langle \Sigma,\Phi\rangle_\HH:=[\langle \sigma,\phi\rangle_\HH]_{\sigma \in \Sigma,\phi \in \Phi}$.

We use the symbol $\N_0$ to denote $\{0,1,\cdots\}$.

\section{Theory on biorthogonal wavelet bases} \label{Stheory}
Let $\VV, \tilde{\VV}, \HH$ be separable Hilbert spaces with $\VV, \tilde{\VV} \hookrightarrow \HH$ with dense embedding.
Identifying $\HH$ with its dual, we obtain the Gelfand triples $\VV \hookrightarrow \HH \hookrightarrow \VV'$ and $\tilde{\VV} \hookrightarrow \HH \hookrightarrow \tilde{\VV}'$ with dense embeddings.
For $s \in [-1,1]$, we set the interpolation spaces $\VV^s:=[\VV',\VV]_{\frac12 s+\frac12}$ and $\tilde{\VV}^s:=[\tilde{\VV}',\tilde{\VV}]_{\frac12 s+\frac12}$.
The following theorem is a special case of an even more general result proven in \cite{53}.

\begin{theorem}[Biorthogonal space decompositions, \cite{56}] \label{main}
Consider two \emph{multiresolution analyses}
\begin{align*}
V_0 \subset V_1 \subset \cdots \subset \HH, \text{ with }
  \clos_{\HH}(\cup_{j \geq 0}V_j) = \HH,\\
  \tilde{V}_0 \subset \tilde{V}_1 \subset \cdots \subset \HH, \text{ with }
 \clos_{\HH}(\cup_{j \geq 0}\tilde{V}_j) = \HH. 
 \end{align*}
 Suppose that for $j \geq 0$ there exist uniformly bounded biorthogonal projectors $Q_j \in \cL(\HH,\HH)$ such that
 \be \label{stab}
 \ran Q_j = V_j,\quad \ran (I-Q_j)=\tilde{V}_j^{\perp_{\HH}},
\ee
  and that, for some $\rho,\,\tilde{\rho}>1$,
 \be \label{Jackson}
 \begin{split}
 \inf_{v_j \in V_j} \|v-v_j\|_{\HH} & \lesssim \rho^{-j} \|v\|_{\VV} \quad(v \in \VV),\\
 \inf_{\tilde{v}_j \in \tilde{V}_j} \|\tilde{v}-\tilde{v}_j\|_{\HH} &\lesssim \tilde{\rho}^{-j} \|\tilde{v}\|_{\tilde{\VV}} \quad(\tilde{v} \in \tilde{\VV}),
 \end{split}
 \ee
 and
 \be \label{Bernstein}
\|v_j\|_{\VV} \lesssim \rho^j \|v_j\|_{\HH}\quad(v_j \in V_j),\quad
\|\tilde{v}_j\|_{\tilde{\VV}} \lesssim \tilde{\rho}^j \|\tilde{v}_j\|_{\HH}\quad(\tilde{v} \in \tilde{V}_j).
\ee
Then,  with $Q_{-1}:=0$, for every $s \in (-1,1)$ it holds that 
$$
 \|v\|_{\VV^s}^2 \eqsim \displaystyle \sum_{j=0}^{\infty}
 \rho^{2js} \|(Q_{j}-Q_{j-1})v\|_{\HH}^2 \quad  (v \in \VV^s).
 $$
  \end{theorem}
  
  \begin{remark}[e.g. \cite{249.76}] \label{rem1} Existence of the biorthogonal projector $Q_j$ as in \eqref{stab} is equivalent to
  \be \label{infsup}
 \beta_j:= \inf_{0 \neq v_j \in V_j} \sup_{0 \neq \tilde{v}_j \in \tilde{V}_j} \frac{\langle v_j,\tilde{v}_j \rangle_\HH}{\|v_j\|_{\HH}\|\tilde{v}_j\|_{\HH}}=
  \inf_{0 \neq \tilde{v}_j \in \tilde{V}_j} \sup_{0 \neq v_j \in V_j} \frac{\langle v_j,\tilde{v}_j \rangle_\HH}{\|v_j\|_{\HH}\|\tilde{v}_j\|_{\HH}}>0,
  \ee
  where $\|Q_j\|_{\cL(\HH,\HH)}^{-1}= \beta_j$.
  
  Other equivalent conditions are that for any Riesz basis $\Phi_j$ for $V_j$ there exists a (unique) $\HH$-dual Riesz basis $\tilde{\Phi}_j$ for $\tilde{V}_j$, and that there exist some Riesz bases $\Phi_j$ and $\tilde{\Phi}_j$ for $V_j$ and $\tilde{V}_j$, respectively, such that $\langle \Phi_j,\tilde{\Phi}_j\rangle_\HH$ is bounded invertible.
In the latter case, one verifies that
$$
\beta_j \geq \frac{\|\langle\Phi_j,\tilde{\Phi}_j\rangle^{-1}_\HH \|^{-1}}{\sqrt{\|\langle\Phi_j,\Phi_j\rangle_\HH \|\|\langle \tilde{\Phi}_j,\tilde{\Phi}_j\rangle_\HH\|}}.
$$
  \end{remark}
  
  \begin{corollary} \label{corol1} In the situation of Thm.~\ref{main},
   for $j \geq 0$ let  $\Psi_j=\{\psi_{j,x}\colon x \in J_j\}$ be a \emph{uniform} Riesz basis for $\ran(Q_j-Q_{j-1})=V_j \cap \tilde{V}_{j-1}^{\perp_\HH}$ $(\tilde{V}_{-1}:=\{0\})$, i.e., with
  $\kappa_\HH(\Psi_j):=\|\langle \Psi_j,\Psi_j\rangle_\HH\| \|\langle \Psi_j,\Psi_j\rangle^{-1}_\HH\|$ it holds that $\sup_j \kappa_\HH(\Psi_j)<\infty$.
  Then for $s \in (-1,1)$, 
  $$
 \bigcup_{j=0}^{\infty} \rho^{-js} \Psi_j \text{ is a Riesz basis for } \VV^s.
 $$

In particular, with $\kappa_s$ denoting the quotient of the supremum and infimum over $v \in \VV_s$ of   $\|v\|_{\VV^s}^2 / \sum_{j=0}^{\infty}
 \rho^{2js} \|(Q_{j}-Q_{j-1})v\|_{\HH}^2$ , it holds that
 $$
 \kappa_{\VV^s}\Big(\bigcup_{j=0}^{\infty} \rho^{-js} \Psi_j\Big) \leq \kappa_s \times  \sup_{j \geq 0} \kappa_\HH(\Psi_j).
 $$
 \end{corollary}

 Recalling that $\ran(Q_0-Q_{-1})=V_0$, the remaining challenge is the construction of $\Psi_j$ for $j \geq 1$, i.e. $\Psi_{j+1}$ for $j \geq 0$:
 \begin{proposition}[\cite{239.17}] \label{construction}
 For $j \geq 0$, let $\Theta_j \cup \Sigma_{j+1}$ and $\tilde{\Phi}_j$ be uniform Riesz bases for $V_{j+1}$ and $\tilde{V}_j$, respectively, such that $\langle \Theta_j,\tilde{\Phi}_j \rangle_\HH=\identity$.
 Then 
 \be \label{formula}
 \Psi_{j+1}=\Xi_{j+1}-\langle \Xi_{j+1},\tilde{\Phi}_j\rangle_\HH \Theta_j
 \ee
is a uniform Riesz basis for $V_{j+1} \cap \tilde{V}_{j}^{\perp_{\HH}}=\ran(Q_{j+1}-Q_{j})$.
 
 With 
 \begin{align*}
 \delta_j&:=\inf_{0 \neq \hat{v}_j \in  \Span \Theta_j} \sup_{0 \neq \tilde{v}_j \in \tilde{V}_j} \frac{\langle \hat{v}_j,\tilde{v}_j \rangle_\HH}{\|\hat{v}_j\|_{\HH}\|\tilde{v}_j\|_{\HH}},\\
 \eps_j&:=\sup_{0 \neq \hat{v}_j \in  \Span \Theta_j} \sup_{0 \neq w_{j+1} \in \Span \Xi_{j+1}} \frac{\langle \hat{v}_j,w_{j+1}\rangle_\HH}{\|\hat{v}_j\|_{\HH}\|w_{j+1}\|_{\HH}},
 \end{align*}
 it holds that
 $$
 \kappa_\HH(\Psi_{j+1}) \leq \frac{(1+\delta_j^{-1})}{\sqrt{1-\eps_j}} \kappa_\HH(\Xi_{j+1}).
 $$
  \end{proposition}
 
 \begin{remark}[dual wavelets] \label{duals} With $(\,)^*$ denoting the adjoint w.r.t. $\langle\,,\,\rangle_\HH$, under the conditions of Thm~\ref{main} equivalently it holds that for every $s \in (-1,1)$,
$$
 \|\tilde{v}\|_{\tilde{\VV}^s}^2 \eqsim \displaystyle \sum_{j=0}^{\infty}
 \rho^{2js} \|(Q^*_{j}-Q^*_{j-1})\tilde{v}\|_{\HH}^2 \quad  (\tilde{v} \in \tilde{\VV}^s),
 $$
 so that for $\tilde{\Psi}_j$ being \emph{any} uniform Riesz basis for $\ran(Q^*_j-Q^*_{j-1})=\tilde{V}_{j+1} \cap V_j^{\perp_\HH}$,
   $$
 \bigcup_{j=0}^{\infty} \rho^{-js} \tilde{\Psi}_j \text{ is a Riesz basis for } \tilde{\VV}^s.
$$

One verifies that the pair $(\ran(Q^*_j-Q^*_{j-1}),\ran(Q_j-Q_{j-1}))$ satisfies \eqref{infsup} with infsup constant $\gamma_j:=\|Q_j-Q_{j-1}\|_{\cL(\HH,\HH)}^{-1} \gtrsim 1$.
Consequently, as stated in Remark~\ref{rem1}, given a Riesz basis $\Psi_j$ for $\ran(Q_j-Q_{j-1})$, there exists a unique dual or biorthogonal Riesz basis $\tilde{\Psi}_j$ for $\ran(Q^*_j-Q^*_{j-1})$, i.e. with $\langle \Psi_j,\tilde{\Psi}_j\rangle_\HH=\identity$.
From
$$
\gamma_j \|{\bf c}_j^\top \tilde{\Psi}_j\|_{\HH} \leq \sup_{0 \neq v_j \in \ran(Q_j-Q_{j-1})} \frac{\langle {\bf c}_j^\top \tilde{\Psi}_j,v_j\rangle_\HH}{\|v_j\|_\HH} \leq \|{\bf c}_j^\top \tilde{\Psi}_j\|_{\HH},
$$
and
$$
\sup_{0 \neq v_j \in \ran(Q_j-Q_{j-1})} \frac{\langle {\bf c}_j^\top \tilde{\Psi}_j,v_j\rangle_\HH}{\|v_j\|_\HH}
=
\sup_{0 \neq {\bf d}_j} \frac{\langle {\bf c}_j^\top \tilde{\Psi}_j,{\bf d}_j^\top \Psi_j\rangle_\HH}{\|v_j\|_\HH}
=
\sup_{0 \neq {\bf d}_j} \frac{\langle {\bf c}_j,{\bf d}_j\rangle}{\|{\bf d}_j\|}\frac{\|{\bf d}_j\|}{\|{\bf d}^\top_j \Psi_j\|_\HH}
$$
one infers that $\kappa_\HH(\tilde{\Psi}_j) \leq \gamma_j^{-1} \kappa_\HH(\Psi_j)$.
Consequently, when the $\Psi_j$ are uniform Riesz bases for $\ran(Q_j-Q_{j-1})$, then their duals are uniform Riesz bases for $\ran(Q^*_j-Q^*_{j-1})$, and for $s \in (-1,1)$, 
$\bigcup_{j=0}^{\infty} \rho^{-js} \Psi_j$ and $\bigcup_{j=0}^{\infty} \rho^{js} \tilde{\Psi}_j$ are $\HH$-biorthogonal Riesz bases for $\VV^s$ and $\tilde{\VV}^{-s}$, respectively.

Moreover, if $\tilde{\VV}=\VV$, then from $\|{\bf c}^\top {\bf D} \Psi\|_{\VV^s}=\sup_{0 \neq {\bf d} \in \ell_2} \frac{\langle {\bf c}^\top  {\bf D} \Psi,{\bf d}^\top  {\bf D}^{-1}\tilde{\Psi}\rangle_\HH}{\|{\bf d}^\top {\bf D}^{-1} \tilde{\Psi}\|_{\tilde{\VV}^{-s}}}=\sup_{0 \neq {\bf d}\in \ell_2} \frac{\langle {\bf c},{\bf d}\rangle}{\|{\bf d}\|} \frac{\|{\bf d}\|}{\|{\bf d}^\top {\bf D}^{-1}\tilde{\Psi}\|_{\tilde{\VV}^{-s}}}$, where ${\bf D} \eqsim \mathrm{blockdiag}[\rho^{-js} \identity]_{j \geq 0}$, and the analogous result with interchanged roles of $(\Psi,{\bf D})$ and $(\tilde{\Psi},{\bf D}^{-1})$, one infers that \linebreak
$\kappa_{\VV^s}(\bigcup_{j=0}^{\infty} \rho^{-js} \Psi_j)=\kappa_{\tilde{\VV}^{-s}}(\bigcup_{j=0}^{\infty} \rho^{js} \tilde{\Psi}_j)$.
\end{remark}

\begin{remark}[dual wavelets cont'd] \label{dualscontd}
An explicit expression for $\tilde{\Psi}_j$ can be obtained in the following special case: 
For $j \geq 0$, let $\Phi_j$ and $\tilde{\Phi}_j$ be $\HH$-biorthogonal Riesz bases for $V_j$ and $\tilde{V}_j$, respectively.
Let ${\bf M}_j=[ {\bf M}_{j,0}\,{\bf M}_{j,1}]$ be the basis transformation from $\Phi_j \cup \Psi_{j+1}$ to $\Phi_{j+1}$, i.e., $[\Phi_j^\top \, \Psi_{j+1}^\top]=\Phi_{j+1}^\top {\bf M}_j$. Analogously, let $[\tilde{\Phi}_j^\top \, \tilde{\Psi}_{j+1}^\top]=\tilde{\Phi}_{j+1}^\top {\bf \tilde{M}}_j$. Biorthogonality shows that
${\bf M}_{j,0}=\langle \tilde{\Phi}_{j+1},\Phi_j\rangle_\HH$, ${\bf M}_{j,1}=\langle \tilde{\Phi}_{j+1},\Psi_{j+1}\rangle_\HH$, and analogous relations at the dual side, as well as ${\bf \tilde{M}}_j={\bf M}_j^{-\top}$.

Now let $\Psi_j$ be constructed as in Prop.~\ref{construction} for the case that $\Theta_j=\Phi_j$ and thus $\hat{V}_j=V_j$.
Then \eqref{formula} reads as ${\bf M}_{j,1}=(\identity-{\bf M}_{j,0} {\bf \tilde{M}}_{j,0}^\top) {\bf R}_{j,1}$, where ${\bf R}_{j,1}=\langle \tilde{\Phi}_{j+1},\Xi_{j+1}\rangle_\HH$. We conclude that
$$
{\bf M}_j=[ {\bf M}_{j,0}\,\,{\bf R}_{j,1}] \left[\begin{array}{@{}cc@{}} \identity & -{\bf \tilde{M}}_{j,0}^\top {\bf R}_{j,1} \\ 0 & \identity\end{array}\right], \text{ i.e. }
{\bf \tilde{M}}_j= \left[\begin{array}{@{}cc@{}} \identity & 0 \\ {\bf R}_{j,1}^\top {\bf \tilde{M}}_{j,0} & \identity\end{array}\right] [ {\bf M}_{j,0}\,\,{\bf R}_{j,1}]^{-\top},
$$
meaning that dual wavelets become explicitly available in terms of $\tilde{\Phi}_{j+1}$ when additionally $\Xi_{j+1}$ is chosen such that the basis transformation $[{\bf M}_{j,0}\,\,{\bf R}_{j,1}]^{-1}$ from $\Phi_{j+1}$ to the two-level basis $\Phi_j \cup \Xi_{j+1}$ is explicitly available.

Classical `stationary' wavelet constructions on the line provide explicitly given (primal and) dual wavelets.
Explicit knowledge of dual wavelets is crucial for applications as data compression and data analysis.
For applications as preconditioning and the (adaptive) solving of operator equations, however, dual wavelets do not enter the computation,
and wavelet constructions on general, non-rectangular domains usually do not provide them (an exception is \cite{249.76}). Allowing $\hat{V}_j \neq V_j$, thus giving up explicit knowledge of dual wavelets, gives an enormous additional freedom in the construction of $\Psi_{j+1}$ in Prop.~\ref{construction}, which we will also exploit in the current work.
\end{remark}

\section{Construction of quadratic Lagrange finite element wavelets} \label{Sconstruction}
\subsection{Multi-resolution analyses}
Given a conforming triangulation $\tria_0$ of a polygon $\Omega \subset \R^2$, let $(\tria_j)_{j \geq 0}$ be the sequence of triangulations where $\tria_{j+1}$ is created from $\tria_j$ by subdividing each triangle $T \in \tria_j$ into four sub-triangles by connecting the midpoints of the edges of $T$ (\emph{red-refinement}).

For $\Gamma \subset \partial\Omega$ being a union of closed edges of triangles $T \in \tria_0$, we define $V_j$ ($\tilde{V}_j$) as the space of continuous piecewise quadratics (linears) w.r.t. $\tria_j$ ($\tria_{j+1}$) that vanish at $\Gamma$. We let ${\mathcal N}(\tria_j)$ denote the set of vertices that are not on $\Gamma$ of $T \in \tria_j$.

We take $\HH:=L_2(\Omega)$, and for some $t \in [1,\frac{3}{2})$, define $\VV=\tilde{\VV}:=H^1_{0,\Gamma}(\Omega) \cap H^t(\Omega)$.
With these definitions, the Jackson and Bernstein estimates \eqref{Jackson}--\eqref{Bernstein} are satisfied with $\rho=2^t$.
After having equipped $V_j$ and $\tilde{V}_j$ with uniform Riesz bases, later in Sect.~\ref{Sinfsup}, with the aid of Remark~\ref{rem1} we will verify also the remaining condition \eqref{stab} of Theorem~\ref{main}.
For $|s|\leq 2$, we define
$$
\ZZ^s\!:=\!\big[(H^1_{0,\Gamma}(\Omega) \cap H^2(\Omega))',H^1_{0,\Gamma}(\Omega) \cap H^2(\Omega)\big]_{\frac{s}{4}+\frac{1}{2}}\!\simeq\!
\left\{ \begin{array}{@{}c@{}l} 
H^1_{0,\Gamma}(\Omega) \cap H^s(\Omega) &\,\, s \in [1,2],\\
 H^s_{0,\Gamma}(\Omega)& \,\,s \in (\frac{1}{2},1),\\
 {[L_2(\Omega),H^1_{0,\Gamma}(\Omega)]_{\frac{1}{2}}}& \,\,s=\frac{1}{2},\\
H^s(\Omega) & \,\,s \in [0,\frac{1}{2}),\\
(\ZZ^{-s})' &\,\, s\in [-2,0).
\end{array}
\right.
$$
Thanks to the \emph{reiteration theorem} we have that $\VV^s\simeq \ZZ^{s t}$. Since $t \in [1,\frac{3}{2})$ was arbitrary, from Corollary~\ref{corol1} we conclude that if $\Psi_j$ is a uniform $L_2(\Omega)$-Riesz basis for $V_j \cap \tilde{V}_{j-1}^{\perp_{L_2(\Omega)}}$, then for $|s|<\frac{3}{2}$,
$$
\bigcup_{j=0}^\infty 2^{-js} \Psi_j \text{ is a Riesz basis for } \ZZ^s.
$$

From the approximation properties of $(V_j)_j$, we infer that for $j \geq 0$, $\psi_{j+1,x} \in \Psi_{j+1}$, $p \in [1,\infty]$ and $\frac{1}{p}+\frac{1}{q}=1$,
\begin{align*}
|\langle u,\psi_{j+1,x}&\rangle_{L_2(\Omega)}| \leq \|\psi_{j+1,x}\|_{L_p(\Omega)} \inf_{v_j \in V_j} \|u-v_j\|_{L_q(\supp \psi_{j,x})}\\
& \lesssim \|\psi_{j+1,x}\|_{L_p(\Omega)} (2^{-j})^2 |u|_{W_q^2(\mathrm{conv\, hull}(\supp \psi_{j,x}))} \quad (u \in W_q^2(\Omega) \cap H^1_{0,\Gamma}(\Omega)),
\end{align*}
that is, the wavelets will have \emph{cancellation properties of order $2$}. In particular, away from the Dirichlet boundary $\Gamma$, they have 2 \emph{vanishing moments}.

\subsection{Local-to-global basis construction}
Following \cite{56,239.17}, in this subsection it will be shown how to reduce the construction of the various collections of functions on $\Omega$, that are needed for the construction of the biorthogonal wavelet basis corresponding to the primal and dual multi-resolution analyses $(V_j)_j$ and $(\tilde{V}_j)_j$, to the construction of corresponding collections of `local' functions on the reference triangle
$$
\bm{T} = \{\lambda \in \R^{3} \colon  \sum_{i=1}^{3} \lambda_i=1, \lambda_i \geq 0\}.
$$
For $1 \leq i \leq 3$, let $\bm{T}_i=\{\lambda \in \bm{T}\colon \lambda_i \leq \frac{1}{2}\}$, and let  $\bm{T}_4=\overline{\bm{T}\setminus \cup_{i=1}^3 \bm{T}_i}$ (\emph{red-refinement}).

For any closed triangle $T$, let $\lambda_T(x) \in \bm{T}$ denote
the barycentric coordinates of $x \in T$ with respect to the
set of vertices of $T$ ordered in some way.

We consider finite collections
of functions $\bm{\Sigma}=\{\bm{\sigma}_{\lambda}:
\lambda \in \bm{I}_{\bm{\Sigma}}\}$ on $\bm{T}$ that satisfy
\begin{enumerate}
\item [($\EuScript{V}$)]
  $\bm{\sigma}_{\lambda}$
  vanishes on any edge or vertex that does not include $\lambda$,
\item [($\EuScript{S}$)]
  $\pi(\bm{I}_{\bm{\Sigma}}) = \bm{I}_{\bm{\Sigma}}$
  and
  $\bm{\sigma}_{\lambda} = \bm{\sigma}_{\pi(\lambda)} \circ \pi$
  for any permutation $\pi:\R^{3} \rightarrow \R^{3}$,
\item [($\EuScript{I}$)]
$\bm{\Sigma}$ is an independent collection of continuous functions.
\end{enumerate}

Such collections of local functions can be used to assemble
collections of global functions in a way known from finite
element methods:
For $j \geq 0$ and with 
\begin{equation}
\label{constrij}
  I_{\Sigma_j} := \{x \in \bar{\Omega}\setminus\Gamma: \lambda_T(x) \in \bm{I}_{\bm{\Sigma}}
   \text{ for some } T \in \tau_{j}\},
\end{equation}
we define the collection $\Sigma_j = \{\sigma_{j,x}\colon x \in I_{\Sigma_j}\}$
of functions on $\Omega$ by
\begin{equation}
\label{constrdef}
  \sigma_{j,x}(y) =
  \left\{ \begin{array}{cl} \mu(x;\tau_j) \bm{\sigma}_{\lambda_T(x)}
  (\lambda_T(y)) & \text{if } x,y \in T \in \tau_{j} \\
  0 & \mbox{otherwise} \end{array} \right.
\end{equation}
with scaling factor $\mu(x;\tau_j) :=
\big( \sum_{\{T \in \tau_j \colon T \ni x\}}
\vol(T)\big)^{-\frac{1}{2}}$.
Note that
the assumptions
($\EuScript{V}$), ($\EuScript{S}$) and ($\EuScript{I}$) show
that $\Sigma_j$ are independent collections of well-defined, continuous functions on $\Omega$.

\begin{lemma} \label{lem1}
Let $\bm{\Sigma}$ and $\bm{\tilde{\Sigma}}$ be two collections of `local'
functions on $\bm{T}$ both satisfying
$(\EuScript{V})$, $(\EuScript{S})$ and $(\EuScript{I})$.
Let $\Sigma_j$ and $\tilde{\Sigma}_j$ denote the corresponding
collections of `global' functions on $\Omega$. Then:
\renewcommand{\theenumi}{\roman{enumi}}
\begin{enumerate}
\item \label{rel1}
$
\|\langle \Sigma_j,\Sigma_j \rangle_{L_2(\Omega)}\| \leq \big\|\frac{\langle \bm{\Sigma},\bm{\Sigma} \rangle_{L_2(\bm{T})}}{\vol(\bm{T})}\big \|,\quad
\|\langle \Sigma_j,\Sigma_j \rangle^{-1}_{L_2(\Omega)}\| \leq \big\|\Big(\frac{\langle \bm{\Sigma},\bm{\Sigma} \rangle_{L_2(\bm{T})}}{\vol(\bm{T})}\Big)^{-1}
\big\|.
$
\item \label{rel2} If 
$\bm{I}_{\bm{\Sigma}} = \bm{I}_{\bm{\tilde{\Sigma}}}$, and 
$\mu \identity \leq \frac{\langle \bm{\Sigma},\bm{\tilde{\Sigma}} \rangle_{L_2(\bm{T})}}{\vol(\bm{T})} \leq M \identity$
then  
$\mu \identity \leq  \langle\Sigma_j,\tilde{\Sigma}_j \rangle_{L_2(\Omega)}  \leq M \identity$.
\item \label{rel3}
$\displaystyle 
\sup_{\begin{array}{@{}c@{}} \scriptstyle 0\neq v_j \in \Span \Sigma_j \vspace*{-0.5ex}\\ \scriptstyle  0\neq \tilde{v}_j \in \Span \tilde{\Sigma}_j\end{array}}
\frac{\langle v_j,\tilde{v}_j\rangle_{L_2(\Omega)}}{\|v_j\|_{L_2(\Omega)}\|\tilde{v}_j\|_{L_2(\Omega)}}
\leq
\sup_{\begin{array}{@{}c@{}} \scriptstyle 0\neq v \in \Span \bm{\Sigma}\vspace*{-0.5ex}\\ \scriptstyle  0\neq \tilde{v} \in \Span \bm{\tilde{\Sigma}} \end{array}}
\frac{\langle v,\tilde{v}\rangle_{L_2(\bm{\Sigma})}}{\|v\|_{L_2(\bm{\Sigma})}\|\tilde{v}\|_{L_2(\bm{\Sigma})}}.
$
\end{enumerate}
\end{lemma}

\begin{proof} We note that for ${\bf \tilde{c}}_j \in \ell_2(I_{\tilde{\Sigma}_j})$,  ${\bf c}_j \in \ell_2(I_{\Sigma_j})$, it holds that
\begin{align*}
&\big\langle \langle \Sigma_j,\tilde{\Sigma}_j \rangle_{L_2(\Omega)} {\bf \tilde{c}}_j, {\bf c}_j\big\rangle=
\langle {\bf c}_j^\top \Sigma_j,{\bf \tilde{c}}_j^\top \tilde{\Sigma}_j \rangle_{L_2(\Omega)}
\\
&=\sum_{T \in \tria_j} \frac{\vol(T)}{\vol(\bm{T})}  \Big\langle \sum_{x \in I_{\Sigma_j} \cap T} c_{j,x} \mu(x;\tria_j) \bm{\sigma}_{\lambda_T(x)},\sum_{y \in I_{\tilde{\Sigma}_j} \cap T} \tilde{c}_{j,y} \mu(y;\tria_j) \bm{\tilde{\sigma}}_{\lambda_T(y)} \Big\rangle_{L_2(\bm{T})}
\\
&=\sum_{T \in \tria_j} \big\langle \langle \bm{\Sigma}, \bm{\tilde{\Sigma}}\rangle_{L_2(\bm{T})} {\bf \tilde{c}}_T,  {\bf c}_T\big\rangle,
\end{align*}
where ${\bf \tilde{c}}_T:=\Big(\tilde{c}_{j,y} \mu(y;\tria_j) \sqrt{\vol(T)}\Big)_{y \in I_{\tilde{\Sigma}_j} \cap T}$,
${\bf c}_T:=\Big(c_{j,x} \mu(x;\tria_j) \sqrt{\vol(T) }\Big)_{x \in I_{\Sigma_j} \cap T}$, 
and 
\begin{align*}
\sum_{T \in \tria_j} \|{\bf \tilde{c}}_T\|^2&=\sum_{T \in \tria_j} \sum_{y \in I_{\tilde{\Sigma}_j} \cap T} |\tilde{c}_{j,y}|^2 \mu(y;\tria_j)^2 \vol(T)
\\ &=
\sum_{y \in I_{\tilde{\Sigma}_j}} |\tilde{c}_{j,y}|^2 \mu(y;\tria_j)^2
\sum_{\{T \in \tria_j\colon T \ni y\}} \vol(T)=\|{\bf \tilde{c}}_j\|^2,
\end{align*}
and similarly $\sum_{T \in \tria_j} \|{\bf c}_T\|^2=\|{\bf c}_j\|^2$.

These relations show \eqref{rel2} and \eqref{rel1}, for the latter using that for $A=A^\top>0$, $\|A\|=\sup_{0\neq x}\frac{\langle A x,x\rangle}{\|x\|^2}$, $\|A^{-1}\|^{-1}=\inf_{0\neq x}\frac{\langle A x,x\rangle}{\|x\|^2}$.

Denoting the value of the supremum at the right hand side of \eqref{rel3} as $Q$, we have
\begin{align*}
\langle v_j,\tilde{v}_j\rangle_{L_2(\Omega)} &=
\sum_{T \in \tria_j} \frac{\vol(T)}{\vol(\bm{T})} \langle v_j \circ \lambda_T^{-1},\tilde{v}_j \circ \lambda_T^{-1}\rangle_{L_2(\bm{T})}\\
& \leq Q \sum_{T \in \tria_j} \frac{\vol(T)}{\vol(\bm{T})} \|v_j \circ \lambda_T^{-1}\|_{L_2(\bm{T})} \|\tilde{v}_j \circ \lambda_T^{-1}\|_{L_2(\bm{T})}\\
& \leq Q \Big[\sum_{T \in \tria_j} \frac{\vol(T)}{\vol(\bm{T})}  \|v_j \circ \lambda_T^{-1}\|_{L_2(\bm{T})}^2\Big]^{\frac{1}{2}} \Big[ \sum_{T \in \tria_j} \frac{\vol(T)}{\vol(\bm{T})}  \|\tilde{v}_j \circ \lambda_T^{-1}\|_{L_2(\bm{T})}^2\Big]^{\frac{1}{2}}\\
&= Q \|v_j\|_{L_2(\Omega)}\|\tilde{v}_j\|_{L_2(\Omega)},
\end{align*}
which completes the proof.
\end{proof}

\subsection{Verification of the uniform inf-sup conditions for $(V_j,\tilde{V}_j)_{j\geq 0}$} \label{Sinfsup}
As a first application of the local-to-global basis construction discussed in the previous subsection, we will verify the existence of the uniform bounded biorthogonal projectors from \eqref{stab}.

Setting $\bm{I}_0:=\N_0^3 \cap \bm{T}$, and for $i \geq 1$, $\bm{I}_i:=(2^{-i} \N_0^3 \cap \bm{T}) \setminus \bm{I}_{i-1}$, we let
$\bm{N}:=\{n_\lambda\colon \lambda \in \bm{I}_0 \cup \bm{I}_1\}$ and  $\bm{\tilde{N}}:=\{\tilde{n}_\lambda\colon \lambda \in \bm{I}_0 \cup \bm{I}_1\}$ denote the standard nodal bases of $P_2(\bm{T})$ and $C(\bm{T}) \cap \prod_{i=1}^4 P_1(\bm{T}_i)$, respectively.

The with $\bm{N}$ and $\bm{\tilde{N}}$ corresponding  `global' collections $N_j$ and $\tilde{N}_j$, defined by \eqref{constrij}--\eqref{constrdef}, span the spaces $V_j$ and $\tilde{V}_j$, respectively. 
The index sets of these collections satisfy $I_{N_j}=I_{\tilde{N}_j}={\mathcal N}(\tria_{j+1})$. Lemma~\ref{lem1}\eqref{rel1} shows that $N_j$ and $\tilde{N}_j$ are uniform $L_2(\Omega)$-Riesz bases for their spans. 
With $\bm{I}_0 \cup \bm{I}_1$ (and $\bm{I}_2$) numbered as indicated in Figure~\ref{fig1}, 
\begin{figure}[h]
\begin{center}
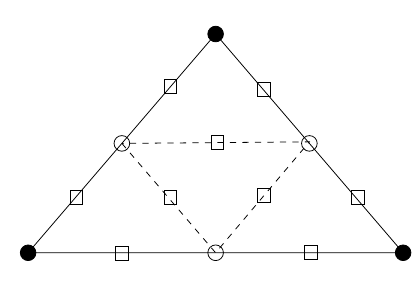
\end{center}
\caption{Numbering of $\bm{I}_0 \cup \bm{I}_1 \cup \bm{I}_2$.}
\label{fig1}
\end{figure}
a direct computation shows that
$$
\frac{\langle \bm{N},\bm{\tilde{N}}\rangle_{L_2(\bm{T})}}{\vol(\bm{T})} = \frac{1}{480} \left[%
\begin{array}{@{}rrrrrr@{}}%
16 & -3 & -3 & -10 & 0 & 0\\
-3 & 16 & -3 & 0 & -10 & 0\\
-3 & -3 & 16 & 0 & 0 & -10\\
2 & 14 & 14 & 70 & 30 & 30\\
14 & 2 & 14 & 30 & 70 & 30\\
14 & 14 & 2 & 30 & 30 & 70
\end{array}\right].
$$
With $\lambda:=\lambda_{\min}\big(\frac{1}{2\vol(\bm{T})} (\langle \bm{N},\bm{\tilde{N}}\rangle_{L_2(\bm{T})}+\langle \bm{N},\bm{\tilde{N}}\rangle_{L_2(\bm{T})}^\top)\big)$, we have $\frac{\langle \bm{N},\bm{\tilde{N}}\rangle_{L_2(\bm{T})}}{\vol(\bm{T})} \geq \lambda \identity$ and thus, by Lemma~\ref{lem1}\eqref{rel2}, 
$\langle N_j,\tilde{N}_j\rangle_{L_2(\Omega)} \geq \lambda \identity$.
It is easily verified that $\lambda>0$, which shows that $\langle N_j,\tilde{N}_j\rangle_{L_2(\Omega)}$ is invertible with $\sup_j \|\langle N_j,\tilde{N}_j\rangle_{L_2(\Omega)}^{-1}\| \leq \lambda^{-1}<\infty$.
Now Remark~\ref{rem1} shows that the inf-sup condition \eqref{infsup}, and thus equivalently \eqref{stab}, are indeed satisfied.

\subsection{Local collections $\bm{\Theta}$, $\bm{\Xi}$, and $\bm{\tilde{\Phi}}$ underlying the construction of $\Psi_{j+1}$}
In order to construct $\Psi_{j+1}$ for $j \geq 0$ by means of equation \eqref{formula}, we need to specify the collections $\Theta_j$, $\Xi_{j+1}$ and $\tilde{\Phi}_j$ of functions on $\Omega$.
We will construct them from collections $\bm{\Theta}=\{\bm{\theta}_\lambda\colon \lambda \in \bm{I}_0 \cup \bm{I}_1\}$, $\bm{\Xi}=\{\bm{\xi}_\lambda\colon \lambda \in \bm{I}_2\}$, and $\bm{\tilde{\Phi}}=\{\bm{\tilde{\phi}}_\lambda \colon \lambda \in \bm{I}_0 \cup \bm{I}_1\}$, respectively, of functions on $\bm{T}$ that satisfy 
$(\EuScript{V})$, $(\EuScript{S})$, and $(\EuScript{I})$, and for which $\bm{\Theta} \cup \bm{\Sigma}$ and $\bm{\tilde{\Phi}}$ are bases for 
$C(\bm{T}) \cap \prod_{i=1}^4 P_2(\bm{T}_i)$ and $C(\bm{T}) \cap \prod_{i=1}^4 P_1(\bm{T}_i)$, respectively, and $\langle \bm{\Theta}, \bm{\tilde{\Phi}}\rangle_{L_2(\bm{T})}=\vol(\bm{T}) \identity$.
Then by an application of Lemma~\ref{lem1}, we know that, as required, $\Theta_j \cup \Sigma_{j+1}$ and $\tilde{\Phi}_j$ are uniform Riesz bases for $V_{j+1}$ and $\tilde{V}_j$, respectively, and $\langle \Theta_j,\tilde{\Phi}_j\rangle_{L_2(\Omega)}=\identity$.
The index sets of these collections satisfy $I_{\Theta_j}=I_{\tilde{\Phi}_j}={\mathcal N}(\tria_{j+1})$, and $I_{\Sigma_{j+1}}={\mathcal N}(\tria_{j+2}) \setminus {\mathcal N}(\tria_{j+1})$.

We will specify $\bm{\Theta} \cup \bm{\Sigma}$ and $\bm{\tilde{\Phi}}$ in terms of  the usual nodal bases $\bm{N}_f$ and $\bm{\tilde{N}}$ for $C(\bm{T}) \cap \prod_{i=1}^4 P_2(\bm{T}_i)$ and $C(\bm{T}) \cap \prod_{i=1}^4 P_1(\bm{T}_i)$, respectively. 
For completeness, with $\bm{N}_f$ we mean the collection of functions in $C(\bm{T}) \cap \prod_{i=1}^4 P_2(\bm{T}_i)$ that have value $1$ in one of the points of $\bm{I}_0 \cup \bm{I}_1 \cup \bm{I}_2$ and vanish at the others.
Aiming at wavelets that have relatively small supports, we exploit the available freedom in the construction to obtain $\bm{\theta}_\lambda$ and $\bm{\xi}_\lambda$ with small supports
and a matrix $\langle \bm{\Xi},\bm{\tilde{\Phi}}\rangle_{L_2(\bm{T})}$ that is sparsely populated. 

We define $\bm{\tilde{\Phi}}$ as indicated in Figure~\ref{fig2}, and $\bm{{\Theta}}$ as indicated in Figure~\ref{fig3}. Both these collections satisfy $(\EuScript{V})$, $(\EuScript{S})$, and $(\EuScript{I})$, and $\bm{\tilde{\Phi}}$ is a basis for $C(\bm{T}) \cap \prod_{i=1}^4 P_1(\bm{T}_i)$.
\begin{figure}[h]
\begin{center}
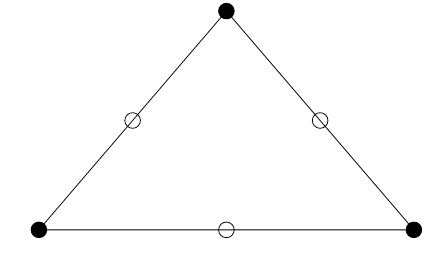\qquad
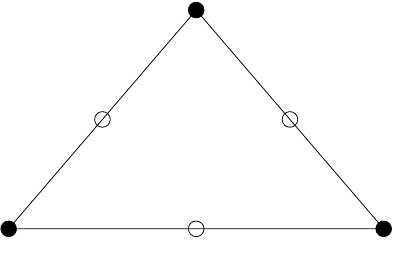
\end{center}
\caption{$\bm{\tilde{\phi}}_{(1,0,0)}$ (left) and $\bm{\tilde{\phi}}_{(\frac12,\frac12,0)}$ (right) in terms of $\bm{\tilde{N}}$. The other $\bm{\tilde{\phi}}_\lambda$ are obtained by permuting the barycentric coordinates.}
\label{fig2}
\end{figure}
\begin{figure}[h]
\begin{center}
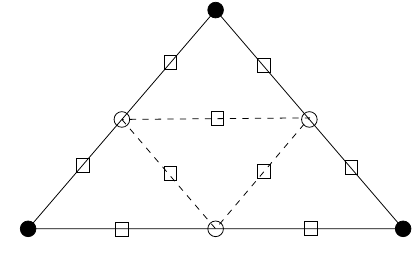 \qquad
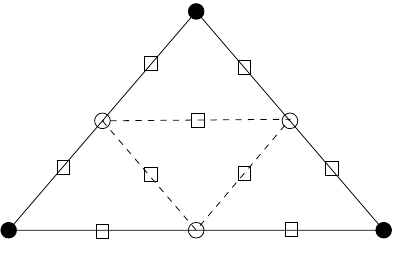
\end{center}
\caption{$\bm{{\theta}}_{(1,0,0)}$ (left) and $\bm{{\theta}}_{(\frac12,\frac12,0)}$ (right) in terms of $\bm{N}_f$. The other $\bm{\theta}_\lambda$ are obtained by permuting the barycentric coordinates.}
\label{fig3}
\end{figure}
A direct computation shows that $\langle \bm{\Theta}, \bm{\tilde{\Phi}}\rangle_{L_2(\bm{T})}=\vol(\bm{T}) \identity$.
We define $\bm{{\Xi}}$ as indicated in Figure~\ref{fig4}. It satisfies $(\EuScript{V})$, $(\EuScript{S})$, and $(\EuScript{I})$,
\begin{figure}[h]
\begin{center}
\raisebox{-2ex}[0pt][0pt]{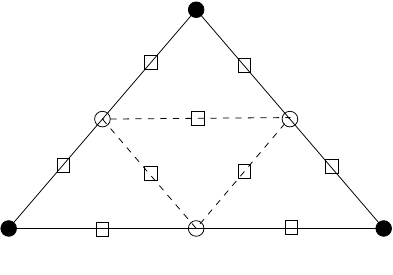} \qquad
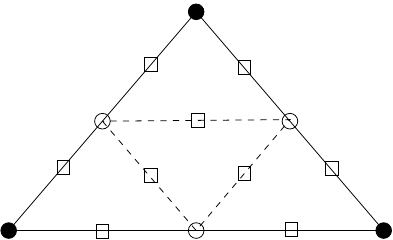
\end{center}
\caption{$\bm{\xi}_{(\frac34,\frac14,0)}$ (left) and $\bm{\xi}_{(\frac14,\frac14,\frac12)}$ (right) in terms of $\bm{N}_f$. The other $\bm{\xi}_\lambda$'s  (five and two) are obtained by permuting the barycentric coordinates.}
\label{fig4}
\end{figure}
and since
\newcommand{\aaa}{\frac{1560}{81}}
\newcommand{\bb}{\frac{530}{81}}
\newcommand{\cc}{\frac{50}{81}}
\newcommand{\qq}{\frac{-12}{25}}
\newcommand{\sss}{\frac{-2}{25}}
\newcommand{\ttt}{\frac{-8}{25}}
\newcommand{\uu}{\frac{1}{25}}
$$ 
\arraycolsep=1.4pt
\def\arraystretch{1.1}
\small
\left|
\begin{array}{@{}ccccccccccccccc@{}} 
72 & 0 & 0 & 0 & 0 & 0 & 0 & 0 & 0 & 0 & 0 & 0 & 0 & 0 & 0  \\
0 & 72 & 0 & 0 & 0 & 0 & 0 & 0 & 0 & 0 & 0 & 0 & 0 & 0 & 0  \\
0 & 0 & 72 & 0 & 0 & 0 & 0 & 0 & 0 & 0 & 0 & 0 & 0 & 0 & 0  \\
0 & 0 & 0 & \aaa & 0 & 0 & \qq & \qq & 0 & 0 & 0 & 0 & 0 & 0 & 0  \\
0 & 0 & 0 & 0 & \aaa & 0 & 0 & 0 & \qq & \qq & 0 & 0 & 0 & 0 & 0  \\
0 & 0 & 0 & 0 & 0 & \aaa & 0 & 0 & 0 & 0 & \qq & \qq & 0 & 0 & 0  \\
0 & 6 & 0 & \bb & 0 & 0 & 1 & \sss & 0 & 0 & 0 & 0 & 0 & 0 & 0  \\
0 & 0 & 6 & \bb & 0 & 0 & \sss & 1 & 0 & 0 & 0 & 0 & 0 & 0 & 0  \\
0 & 0 & 6 & 0 & \bb & 0 & 0 & 0 & 1 & \sss & 0 & 0 & 0 & 0 & 0  \\
6 & 0 & 0 & 0 & \bb & 0 & 0 & 0 & \sss & 1 & 0 & 0 & 0 & 0 & 0  \\
6 & 0 & 0 & 0 & 0 & \bb & 0 & 0 & 0 & 0 & 1 & \sss & 0 & 0 & 0  \\
0 & 6 & 0 & 0 & 0 & \bb & 0 & 0 & 0 & 0 & \sss & 1 & 0 & 0 & 0  \\
0 & 0 & 0 & 0 & \cc & \cc & 0 & 0 & \uu & \ttt & \ttt & \uu & \frac{-5}{4} & 1 & 1  \\
0 & 0 & 0 & \cc & 0 & \cc & \ttt & \uu & 0 & 0 & \uu & \ttt & 1 & \frac{-5}{4}  & 1  \\
0 & 0 & 0 & \cc & \cc & 0 & \uu & \ttt & \ttt & \uu & 0 & 0 & 1 & 1 & \frac{-5}{4} 
\end{array}
\right|
\neq 0,
$$
we conclude that $\bm{\Theta} \cup \bm{\Xi}$ is a basis for $C(\bm{T}) \cap \prod_{i=1}^4 P_2(\bm{T}_i)$.
It holds that
\begin{align*}
 \langle \bm{\xi}_{(\frac34,\frac14,0)},\bm{\tilde{\phi}}_\lambda\rangle_{L_2(\bm{T})}& =\vol(\bm{T}) \times \left\{
\def\arraystretch{1.2}
\begin{array}{cl} 
\frac{3}{100} & \lambda=(1,0,0), \\
0 & \lambda \in (\bm{I}_0 \cup \bm{I}_1)\setminus \{(1,0,0)\},
\end{array}
\right.\\
\langle \bm{\xi}_{(\frac14,\frac14,\frac12)},\bm{\tilde{\phi}}_\lambda\rangle_{L_2(\bm{T})}& =\vol(\bm{T}) \times \left\{
\def\arraystretch{1.2}
\begin{array}{cl} 
\frac{-1}{48} & \lambda=(0,0,1),\\
\frac{27}{240} & \lambda=(\frac12,\frac12,0), \\
0 & \lambda \in (\bm{I}_0 \cup \bm{I}_1)\setminus\{(0,0,1), (\frac12,\frac12,0)\},
\end{array}
\right.
\end{align*}
where the other values of $\langle \bm{\xi}_{\mu},\bm{\tilde{\phi}}_\lambda\rangle_{L_2(\bm{T})}$ for $(\mu,\lambda) \in \bm{I}_2 \times (\bm{I}_1 \cup \bm{I}_0)$ are obtained by permuting the barycentric coordinates.

\subsection{Definition of the $\Psi_j$}
We take $\Psi_0=N_0$, being a Riesz basis for $V_0$.
In the previous subsection, from corresponding local collections we have constructed uniform Riesz bases
$\Theta_j \cup \Sigma_{j+1}$ and $\tilde{\Phi}_j$ for $V_{j+1}$ and $\tilde{V}_j$, respectively, such that $\langle \Theta_j,\tilde{\Phi}_j \rangle_{L_2(\Omega)}=\identity$.
For $j \geq 0$, the collection of wavelets $\Psi_{j+1}$ is now given by the explicit formula $\Psi_{j+1}=\Xi_{j+1}-\langle \Xi_{j+1},\tilde{\Phi}_j\rangle_{L_2(\Omega)} \Theta_j$.
These wavelets will depend on the topology of $\tria_0$ via the the local-to-global  basis construction \eqref{constrij}--\eqref{constrdef} that we applied for the definition of $\Theta_j $, $\Sigma_{j+1}$, and $\tilde{\Phi}_j$,  as well via the inner product $\langle \Xi_{j+1},\tilde{\Phi}_j\rangle_{L_2(\Omega)}$. 

Despite of these dependencies on $\tria_0$, as well as that on $\Gamma$, we can distinguish between two types of wavelets:
A wavelet $\psi_{j+1,x}$ of the first type stems from $\bm{\xi}_\lambda$ of type $\bm{\xi}_{(\frac{3}{4},\frac{1}{4},0)}$ (see Figure~\ref{fig4}), so that $x \in {\mathcal N}(\tria_{j+2}) \setminus {\mathcal N}(\tria_{j+1})$ is on an edge of a $T \in \tria_j$.
It equals $\xi_{j+1,x}$ minus a multiple of \emph{one} (or, near the Dirichlet boundary,  possibly zero) $\theta_{j,y}$ with $y \in {\mathcal N}(\tria_j)$ (left picture in Figure~\ref{fig2}).
A wavelet $\psi_{j+1,x}$ of the second type stems from $\bm{\xi}_\lambda$ of type $\bm{\xi}_{(\frac{1}{4},\frac{1}{4},\frac{3}{4})}$, so that $x \in {\mathcal N}(\tria_{j+2}) \setminus {\mathcal N}(\tria_{j+1})$ is interior to a $T \in \tria_j$.
It equals $\xi_{j+1,x}$ minus a multiple of \emph{two} (or, near the Dirichlet boundary,  possibly one or zero) $\theta_{j,y}$'s, where one $y \in {\mathcal N}(\tria_j)$ and the other is in ${\mathcal N}(\tria_{j+1}) \setminus {\mathcal N}(\tria_j)$ (right picture in Figure~\ref{fig2}).
The supports of both types of wavelets (away from the Dirichlet boundary) are illustrated in Figure~\ref{fig5}.
\begin{figure}[h]
\begin{center}
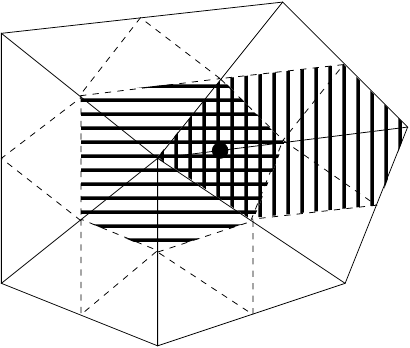\qquad
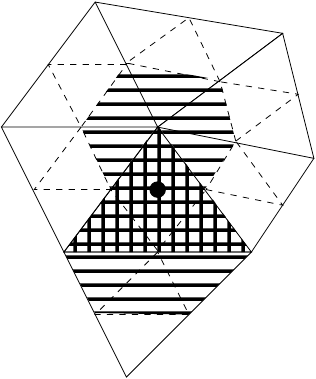
\end{center}
\caption{Wavelets $\psi_{j+1,x}$ for $x \in {\mathcal N}(\tria_{j+2}) \setminus {\mathcal N}(\tria_{j+1})$ on an edge of  a $T \in \tria_j$ (left) or interior to a $T \in \tria_j$ (right).
 Indicated are $x$ ($\bullet$), the support of $\xi_{j+1,x}$ (vertical shading), and that of the $\theta_{j,y}$'s (horizontal shading).
 The support of the wavelet is the union of these supports.
 \newline The wavelets are linear combinations of 11 (left) or 13 (right) quadratic nodal basis functions (i.e. functions from $N_{j+1}$) (the numbers 11 and 13 apply when the valence of each interior vertex in $\tria_j$ is 6).}
\label{fig5}
\end{figure}

\subsection{Condition numbers}
For $\Omega =(0,1)^2$, $\tria_0=\big\{\{(x,y)\colon 0 \leq x \leq y \leq 1\}$, $\{(x,y)\colon 0 \leq y \leq x \leq 1\}\big\}$, and $\Gamma=\partial\Omega$, we have computed $\kappa_{L_2(\Omega)}(\bigcup_{j=0}^J \Psi_j)$ and $\kappa_{H^1_0(\Omega)}(\bigcup_{j=0}^J \Psi_j)$, with $H^1_0(\Omega)$ equipped with $|\cdot|_{H^1(\Omega)}$, and with the wavelets normalized in the corresponding norm.
That is, we have computed the condition numbers of the (normalized) mass and stiffness matrices. The results are given in Tables~\ref{tab1} and \ref{tab2}.

\begin{table}
\caption{$L_2$-condition numbers of the $L_2$-normalized wavelets up to level $J$.}
\begin{tabular}{|c|c|c|c|c|c|c|c|c|c|}
\hline
$J$ & 0 & 1 & 2 & 3 & 4 & 5 & 6 & 7 & 8\\ \hline
$\kappa_{L_2((0,1)^2)}$ & 1 & 4.8 & 7.3 & 8.3 & 8.9 & 9.2 & 9.7 & 9.8 & 9.9\\
\hline
\end{tabular}
\label{tab1}
\end{table}

\begin{table}
\caption{$H_0^1$-condition numbers of the $H_0^1$-normalized wavelets up to level $J$.}
\begin{tabular}{|c|c|c|c|c|c|c|c|c|c|}
\hline
$J$ & 0 & 1 & 2 & 3 & 4 & 5 & 6 & 7 & 8\\ \hline
$\kappa_{H^1_0((0,1)^2)}$ & 1 & 27 & 41 & 54 & 63 & 70 & 76 & 81 & 85\\
\hline
\end{tabular}
\label{tab2}
\end{table}

\begin{remark} The principles behind the finite element wavelet construction applied in this work were introduced in \cite{249.71, 56}. 
The realizations of finite element wavelets for given primal and dual orders from \cite{56} were revisited in \cite{239.17} in order to obtain smaller condition numbers.
The two-dimensional piecewise quadratic wavelets from \cite{239.17} have 3 vanishing moments (since $\tilde{V}_j=V_j$). Their $L_2$- and $H_0^1$-condition numbers on the square on level $J=6$ were $12$ and $60$, respectively, for the $H_0^1$-case somewhat improving upon the condition numbers from Table~\ref{tab2}.
On the other hand, for a regular mesh where all vertices have valence 6, each quadratic wavelet away from the boundary in \cite{239.17} is a linear combination of not less than 87 nodal basis functions, which high number was the motivation for the current work.

Two-dimensional quadratic finite element wavelets where each wavelet is a linear combination of only 4 or 6 nodal basis functions were introduced in \cite{187}.
These wavelets are based on a splitting of $V_{j+1}$ into $V_j$ and an orthogonal complement w.r.t. a `discrete' level-dependent scalar product on $V_{j+1}$.
It was shown that the wavelets generate a Riesz basis for $H^s$ for $s \in (0.3974,\frac{3}{2})$.
These wavelets, however, have no vanishing moment, and, consequently, they cannot be expected to generate a Riesz basis for $H^s$ for $s \leq 0$.
\end{remark}

Finally, since for our application in solving parabolic PDEs it is needed that the wavelets form a Riesz basis in $H^{-1}(\Omega)$, based on $\kappa_{H^{-1}(\Omega)}({\bf D} \Psi)=\kappa_{H^1_0(\Omega)}({\bf D}^{-1} \tilde{\Psi})$ (see Remark~\ref{duals}), where $\tilde{\Psi}$ is the dual wavelet basis and ${\bf D} \eqsim \mathrm{blockdiag}[2^{-j}]_{j \geq 0} \eqsim \mathrm{diag}[|\tilde{\psi}|_{H^1(\Omega)}]_{\tilde{\psi} \in \tilde{\Psi}}$, additionally we have computed $\kappa_{H^1_0(\Omega)}(\bigcup_{j=0}^J \tilde{\Psi}_j)$ with these dual wavelets being normalized w.r.t. $|\cdot|_{H^1(\Omega)}$. The results are given in Table~\ref{tab3}. Since, as explained in Remark~\ref{dualscontd}, these computations involve the applications of inverses of mass matrices and that of two-level transforms at the primal side, which inverses in our case are densely populated, we computed these condition numbers only up to level 6.

\begin{table}
\caption{$H^1$-condition numbers of the $H^1$-normalized \emph{dual} wavelets up to level $J$.}
\begin{tabular}{|c|c|c|c|c|c|c|c|}
\hline
$J$ & 0 & 1 & 2 & 3 & 4 & 5 & 6 \\ \hline
$\kappa_{H^1_0((0,1)^2)}$ & 1 &  6.5 & 14 & 22 & 28 & 32 & 36 \\
\hline
\end{tabular}
\label{tab3}
\end{table}

\bibliographystyle{alpha}
\bibliography{../ref}

\end{document}